\documentclass[reqno]{amsart}

\usepackage{graphicx}
\usepackage{hyperref}
\usepackage{ellipsis}
\usepackage{amsmath,amsthm, amssymb}
\usepackage{enumerate}
\usepackage{comment}
\usepackage[font=small,labelfont=bf]{caption}
\usepackage{subcaption}
\usepackage{url}
\usepackage{cite}
\usepackage{color}
\usepackage{color}
\usepackage{breakurl}
\usepackage{comment}
\usepackage{mathtools}
\mathtoolsset{showonlyrefs=true,showmanualtags=true}
\newcommand{\bburl}[1]{\textcolor{blue}{\url{#1}}}
\usepackage{fixltx2e,amsmath}
\MakeRobust{\eqref}

\usepackage{array}

\usepackage{verbatim}

\numberwithin{part}{section}
\numberwithin{equation}{section}

\newtheorem{lemma}[part]{Lemma}
\newtheorem{question}[part]{Question}
\newtheorem{defn}[part]{Definition}
\newtheorem{rek}[part]{Remark}
\newtheorem{thm}[part]{Theorem}
\newtheorem{conjecture}[part]{Conjecture}
\newtheorem{cor}[part]{Corollary}

\newcommand{\Z}{{\mathbb Z}}
\newcommand{\N}{{\mathbb N}}

\newcommand{\F}{{\mathbb F}}

\title{Subsets of $\F_q[x]$ Free of 3-term Geometric Progressions}

\author[Asada]{Megumi Asada}
\email{\textcolor{blue}{\href{mailto:maa2@williams.edu}{maa2@williams.edu}}}
\address{Department of Mathematics and Statistics, Williams College, Williamstown, MA 01267}

\author[Fourakis]{Eva Fourakis}
\email{\textcolor{blue}{\href{mailto:erf1@williams.edu}{erf1@williams.edu}}}
\address{Department of Mathematics and Statistics, Williams College, Williamstown, MA 01267}

\author[Manski]{Sarah Manski}
\email{\textcolor{blue}{\href{mailto:Sarah.Manski12@kzoo.edu}{Sarah.Manski12@kzoo.edu}}}
\address{Department of Mathematics \& Computer Science, Kalamazoo College, Kalamazoo, MI 49006}

\author[McNew]{Nathan McNew}
\email{\textcolor{blue}{\href{mailto:nmcnew@towson.edu}{nmcnew@towson.edu}}}
\address{Department of Mathematics, Towson University, Towson, MD 21252}

\author[Miller]{Steven J. Miller}
\email{\textcolor{blue}{\href{mailto:sjm1@williams.edu}{sjm1@williams.edu}},  \textcolor{blue}{\href{Steven.Miller.MC.96@aya.yale.edu}{Steven.Miller.MC.96@aya.yale.edu}}}
\address{Department of Mathematics and Statistics, Williams College, Williamstown, MA 01267}

\author[Moreland]{Gwyneth Moreland}
\email{\textcolor{blue}{\href{mailto:gwynm@umich.edu}{gwynm@umich.edu}}}
\address{Department of Mathematics, University of Michigan, Ann Arbor, MI 48109}

\date{\today}



\begin{document}

\thanks{This research was conducted as part of the SMALL 2015 REU program. The authors were supported by Williams College, NSF grants DMS1347804 and DMS1265673, and the first named author was additionally supported by a Clare Boothe Luce scholarship. We thank our colleagues from the 2015 Williams SMALL REU for many helpful conversations. }

\begin{abstract}
Several recent papers have considered the Ramsey-theoretic problem of how large a subset of integers can be without containing any 3-term geometric progressions.  This problem has also recently been generalized to number fields, determining bounds on the greatest possible density of ideals avoiding geometric progressions. We study the analogous problem over $\F_q[x]$, first constructing a set greedily which avoids these progressions and calculating its density, and then considering bounds on the upper density of subsets of $\F_q[x]$ which avoid 3-term geometric progressions. This new setting gives us a parameter $q$ to vary and study how our bounds converge to 1 as it changes, and finite characteristic introduces some extra combinatorial structure that increases the tractibility of common questions in this area.
\end{abstract}

\maketitle

\tableofcontents

%%%%%%%%%%%%%%%%%%%%%%%%%%%%%%%%%%%%%%%%%%%%%%%%%%%%%%%%%%%%%%%%%%%%%%%%%%%%%%%%%%%%%%%%%%%%%%%%%%%%%%%%%%%%%%%%%%%%%%%%%%%%%%%%%%%%%%
%%%%%%%%%%%%%%%%%%%%%%%%%%%%%%%%%%%%%%%%%%%%%%%%%%%%%%%%%%%%%%%%%%%%%%%%%%%%%%%%%%%%%%%%%%%%%%%%%%%%%%%%%%%%%%%%%%%%%%%%%%%%%%%%%%%%%%
%%%%%%%%%%%%%%%%%%%%%%%%%%%%%%%%%%%%%%%%%%%%%%%%%%%%%%%%%%%%%%%%%%%%%%%%%%%%%%%%%%%%%%%%%%%%%%%%%%%%%%%%%%%%%%%%%%%%%%%%%%%%%%%%%%%%%%
\section{Introduction}

In a 1961 paper Rankin \cite{Ran} introduced the idea of considering how large a set of integers can be without containing terms which are in geometric progression.  He constructed a subset of the integers which avoids 3-term geometric progressions and has asymptotic density approximately 0.719745.  Brown and Gordon \cite{BG} noted that the set Rankin considered was the set obtained by greedily including integers subject to the condition that such integers do not create a progression involving integers already included in the set.

Other authors, including Riddell \cite{Rid}, Biegelb\"{o}ck, Bergelson, Hindman and Strauss \cite{BBHS}, Nathanson and O'Bryant \cite{NO}, and McNew \cite{Mc}, have refined bounds for the upper density of a set which avoids geometric progressions.  Best, Huan, McNew, Miller, Powell, Tor and Weinstein \cite{BHMMPTW} generalized the problem to quadratic number fields. Using many of the techniques from these other works, they obtained similar results for the density of the ideals in the ring of integers which similarly avoid geometric progressions.

The purpose of \cite{BHMMPTW}  was to see how the results differed when considering subsets of number fields rather than the integers.  Here we investigate what happens over function fields of finite characteristic.  In particular, using combinatorial tools as well as the methods of Rankin, McNew, and Best et al., we consider the size of the largest subset of the polynomial ring $\F_q[x]$ which avoids geometric progressions whose common ratio is a non-unit polynomial in this ring. 

\begin{rek} It is worth remarking on the choice of problem. In the integer case it is interesting to study sets which avoid 3-term geometric progressions with integral ratio as these sets have a very different flavor from those sets (called primitive sets) which avoid 2-term progressions with integral ratio. The situation is richer over $\F_q[x]$, for example we now have many more units.  In what follows we will require that our sets avoid 3-terms in a non-unit progression.
\end{rek}

We begin by constructing and characterizing the set constructed greedily (with respect to norm) which avoids 3-term non-unit geometric progressions.  In Theorem~\ref{thm:greedydens} we show that this set has asymptotic density \begin{equation}  \left(1 - \frac{1}{q}\right) \prod_{i = 1}^{\infty} \prod_{n = 1}^{\infty} \left(1 + q^{-3^in}\right)^{m(n,q)}=\left(1 - \frac{1}{q}\right) \prod_{i = 1}^{\infty} \frac{1-q^{1-2\cdot 3^i}}{1-q^{1-3^i}},\end{equation}
where $m(n,q)$ denotes the number of monic irreducible polynomials over $\F_q$ of degree $n$. We then study bounds on the upper density of sets avoiding $3$-term non-unit progressions, and give the numerical values obtained for certain specific values of $q$.

%%%%%%%%%%%%%%%%%%%%%%%%%%%%%%%%%%%%%%%%%%%%%%%%%%%%%%%%%%%%%%%%%%%%%%%%%%%%%%%%%%%%%%%%%%%%%%%%%%%%%%%%%%%%%%%%%%%%%%%%%%%%%%%%%%%%%%
%%%%%%%%%%%%%%%%%%%%%%%%%%%%%%%%%%%%%%%%%%%%%%%%%%%%%%%%%%%%%%%%%%%%%%%%%%%%%%%%%%%%%%%%%%%%%%%%%%%%%%%%%%%%%%%%%%%%%%%%%%%%%%%%%%%%%%
%%%%%%%%%%%%%%%%%%%%%%%%%%%%%%%%%%%%%%%%%%%%%%%%%%%%%%%%%%%%%%%%%%%%%%%%%%%%%%%%%%%%%%%%%%%%%%%%%%%%%%%%%%%%%%%%%%%%%%%%%%%%%%%%%%%%%%
\section{Construction of the Greedy Set}
Our goal is to construct a subset of significant density whose elements do not contain a three-term non-unit geometric progression. That is, a set that does not contain all of $b(x), r(x)b(x)$ and $r(x)^2b(x)$ for some $b(x),r(x) \in \F_q[x]$ where $r(x)$ is not a unit of the ring. Observe that $\F_q[x]$ is a unique factorization domain, and so each $f(x) \in \F_q[x]$ can be writen uniquely as \begin{equation} f(x) \ = \ u P_1^{e_1}(x) \cdots P_k^{e_k}(x),\end{equation}
where $P_1(x),\cdots,P_k(x)$ are monic irreducible polynomials (primes) and $u \in \F_q$ is a unit. 

Let $A_3^*(\Z) \ = \ \{0,1,3,4,9,10,12,13, \dots\}$ denote the set of non-negative integers formed by greedily constructing a set free of 3-term arithmetic progressions (see \cite{Ran}). Rankin observed that the set of integers whose prime factors only have exponents coming from the set $A_3^*(\Z)$, the set obtained by greedily constructing a set to avoid containing any geometric progressions, has a relatively high density, $0.719745$.  We wish to similarly construct a greedy subset of $\F_q[x]$ to avoid geometric progressions, and consider its density.  We will see that the resulting set can be characterized in a way similar to Rankin's set as those elements whose factorizations only have exponents contained in the set $A_3^*(\Z)$. We will refer to this greedy set $G_3^*(\F_q[x])$  as $G_{3,q}^*$ for brevity. Remember that for our construction, unit progressions are ignored.

More precisely, we let $G_{3,q}^*$ be the subset of $\F_q[x]$ formed by taking the set $\F_q$, and greedily adding in elements by increasing degree, subject to the condition that included elements are not part of a 3-term geometric progression involving elements of smaller degree already included in $G_{3,q}^*$.

We will see that, similarly to Rankin's set, the elements of $\F_q[x]$ that are included in $G_{3,q}^*$ are those polynomials with factorizations \begin{equation}
f(x) = uP_1(x)^{e_1} \cdots P_k(x)^{e_k}
\end{equation} where each exponent $e_i$ on a monic irreducible factor $P_i(x)$ is in $A_3^*(\Z)$. In order to show this we must first understand the behavior of the units of this ring. Consider the three terms
\begin{equation} f(x), \ ur(x)f(x), \ vr^2(x)f(x)\end{equation}
 where $u$ and $v$ are units. Suppose $f(x), \ ur(x)f(x) \in G_{3,q}^*$. The exponents appearing on the irreducible factors of $f(x), \ ur(x)f(x)$ and  $vr(x)^2f(x)$ contain an arithmetic progression, and thus we would like to conclude that the element $vr(x)^2f(x)$ must be excluded from the greedy set. However, it is not immediately clear that there exists another ratio $\rho(x)$ and a polynomial $g(x)$ such that $g(x), \rho(x)g(x) \in G_{3,q}^*$ and $\rho(x)^2g(x) = vr(x)^2f(x)$, thus requiring us to exclude this polynomial from the greedy set. Nevertheless such a progression does in fact exist as is seen from the following theorem, which allows us to characterize the greedy set $G_{3,q}^*$ by its elements' irreducible factors.

\begin{thm}\label{characterization}
Let $G^*_{3,q}$ be the greedily constructed subset of $\F_q[x]$ described above. Then the polynomials contained in $G_{3,q}^*$ are precisely those polynomials \begin{equation}
f(x) = uP_1(x)^{e_1} \cdots P_k(x)^{e_k}
\end{equation} where each $P_i(x)$ is monic and irreducible, $u\in \F_q$ is a unit and $e_i \in A_3^*(\Z)$ for all $i$.
\end{thm}

\begin{proof}
We proceed by induction on the degree, $n$ of the polynomials included in $G_{3,q}^*$. Note that for $n \in \{0,1\}$, the statement is true since there are no non-unit geometric progressions among the polynomials of degree at most 1. Now suppose that all polynomials of degree at most $N$ satisfy the inductive hypothesis. Consider $g(x) \ = \ uP_1(x)^{e_1} \cdots P_k(x)^{e_k} \in \F_q[x]$ a polynomial of degree $n=N+1$ with each $P_i(x)$ monic and irreducible and at least one exponent $e_i$ not in $A_3^*(\Z)$. 

For each $e_i \not\in A_3^*(\Z)$ there exist $a_i$, $b_i \in A_3^*(\Z)$ such that $a_i,b_i,e_i$ forms an arithmetic progression with common difference $d_i$. That is $b_i=a_i+d_i$, $e_i=a_i+2d_i$.  Let \[r(x) = \prod_{\substack{i\leq k\\e_i \notin A_3^*(\Z)}} P_i(x)^{d_i}.\]  Now, consider the geometric progression, \begin{equation}
\frac{g(x)}{r(x)^2} \ ,\ \frac{g(x)}{r(x)} \ , \ g(x).  
\end{equation}
Note that each of these terms is contained in $\F_q[x]$, and, by construction, all of the  exponents appearing in the factorizations of the first two terms are in $A_3^*(\Z)$.  They also have degree less than or equal to $n$ and so by the inductive hypotheses, are already contained in $G_{3,q}^*$.  Thus the element $g(x)$ is not included in $G^*_{3,q}$.

Now suppose $h(x) \ = \ uP_1^{e_1}(x) \cdots P_k^{e_k}(x) \in \F_q[x]$ is a polynomial of degree $n$ with each exponent $e_i \in A_3^*(\Z)$. Note that $h(x)$ cannot form a geometric progression with any other two elements of degree at most $N$ in $G_{3,q}^*$, because the exponents on the irreducible factors of any terms in geometric progression would form an arithmetic progression and the polynomials in $G^*_{3,q}$ are constructed so that is not possible.
\end{proof}

%%%%%%%%%%%%%%%%%%%%%%%%%%%%%%%%%%%%%%%%%%%%%%%%%%%%%%%%%%%%%%%%%%%%%%%%%%%%%%%%%%%%%%%%%%%%%%%%%%%%%%%%%%%%%%%%%%%%%%%%%%%%%%%%%%%%%%
%%%%%%%%%%%%%%%%%%%%%%%%%%%%%%%%%%%%%%%%%%%%%%%%%%%%%%%%%%%%%%%%%%%%%%%%%%%%%%%%%%%%%%%%%%%%%%%%%%%%%%%%%%%%%%%%%%%%%%%%%%%%%%%%%%%%%%
%%%%%%%%%%%%%%%%%%%%%%%%%%%%%%%%%%%%%%%%%%%%%%%%%%%%%%%%%%%%%%%%%%%%%%%%%%%%%%%%%%%%%%%%%%%%%%%%%%%%%%%%%%%%%%%%%%%%%%%%%%%%%%%%%%%%%%
\section{Computing Densities}

%%%%%%%%%%%%%%%%%%%%%%%%%%%%%%%%%%%%%%%%%%%%%%%%%%%%%%%%%%%%
%%%%%%%%%%%%%%%%%%%%%%%%%%%%%%%%%%%%%%%%%%%%%%%%%%%%%%%%%%%%
%%%%%%%%%%%%%%%%%%%%%%%%%%%%%%%%%%%%%%%%%%%%%%%%%%%%%%%%%%%%
\subsection{Set-Up}
% note to self: how to make it clear I am doing an ideal norm thing and don't need to consider unit multiples of
% primes because those generate the same ideal?
We now consider the density of the set $G_{3,q}^*$, and obtain a function that gives this density as a function of $q$.  We give first several useful definitions we will need to investigate this density.

% IGNORE for now, use this if we switch back to language of ideals.

\begin{defn} The norm of an element $f(x) \in \F_q[x]$ is
\begin{equation} N( f ) = N(\langle f \rangle) \ = \  |\F_q[x]/ \langle f(x) \rangle| \ = \ q^{\deg f}.\end{equation}
In particular, the norm of an element is the norm of the ideal it generates. Note that the norm is multiplicative.
\end{defn}

\begin{defn}\label{def:normle}
Let $S(n) := \{f(x) \in \F_q[x] : N(f) \le n\}$. Similarly, for ${X \subseteq \N}$ we have $S(X) = {\{f(x) \in \F_q[x] : N(f) \in X\}}$.
\end{defn}
It will also be convenient to have notation to refer to the set of monic irreducible polynomials.
\begin{defn}\label{def:iqn}
Let $I_q := \{f(x) \in \F_q[x] : f(x) \text{ is monic irreducible}\}$, and let $I_{q,n} := \{f(x) \in I_q : \deg f = n\}$.  When we refer to the number of monic irreducible polynomials of degree $n$ we will use the notation $m(q,n) = |I_{q,n}|$.
\end{defn}

\begin{defn}
The (asymptotic) density of a set $A \subseteq \F_q$ is defined as
\begin{equation} \label{Equation:Zeta} d(A) \ = \ \lim_{n \to \infty} \frac{|A \cap S(n)|}{|S(n)|}\end{equation}
(provided that this limit exists).
\end{defn}

For our density formulas, we require a Riemann zeta-type function for $\F_p[x]$.

\begin{defn}\label{zetadef} For $\F_q[x]$, define
\begin{equation} \zeta_q(s) \ := \ \sum_{\substack{f(x) \in \F_q[x]\\f(x) \text{ \emph{monic}}}} \frac{1}{N(f)^{s}}.\end{equation}
This is an analogue of the Riemann zeta function, and likewise has its own Euler product:
\begin{equation}\label{Equation:Eulerprod}
\zeta_q(s) \ = \ \prod_{p(x) \in I_q} \frac{1}{1 - N(p)^{-s}} \ = \ \prod_{p(x) \in I_q} \frac{1}{1 - q^{-s \deg p}}.
\end{equation}
\end{defn}

We can rewrite $\zeta_q(s)$ using a fact from elementary number theory.

\begin{lemma} \label{degreesproduct} We have
\begin{align*}\zeta_q(s) \ = \ \prod_{n = 1}^{\infty} \frac{1}{(1 - q^{-ns})^{m(n,q)}}=\frac{1}{1-q^{1-s}}.
\end{align*}
\end{lemma}

\begin{proof}

The first expression follows from collecting the degree $n$ irreducible polynomials in the Euler product expression in Definition \ref{zetadef}.  The second follows from \eqref{Equation:Eulerprod} and the fact that there are exactly $q^d$ monic polynomials of norm $q^d$.
\end{proof}

%%%%%%%%%%%%%%%%%%%%%%%%%%%%%%%%%%%%%%%%%%%%%%%%%%%%%%%%%%%%
%%%%%%%%%%%%%%%%%%%%%%%%%%%%%%%%%%%%%%%%%%%%%%%%%%%%%%%%%%%%
%%%%%%%%%%%%%%%%%%%%%%%%%%%%%%%%%%%%%%%%%%%%%%%%%%%%%%%%%%%%
\subsection{Density of the Greedy Set}
Let $P(x) \in \F_q[x]$ be a monic, irreducible polynomial. Then, as described in \cite{BHMMPTW}, the density of those polynomials in $\F_q[x]$ exactly divisible by $P(x)^k$, for some  $k \in A_3^*(\Z)$, is
\begin{equation}\left(\frac{N(P)-1}{N(P)}\right) \sum_{i \in A_3^*(\Z)} \frac{1}{N(P)^i} \ = \ \left(1 - \frac{1}{N(P)}\right) \prod_{i = 0}^{\infty} \left(1 + \frac{1}{N(P)^{3^{i}}}\right). \label{Equation:PrimeDensity}
\end{equation}

\begin{thm} \label{thm:greedydens} Define
\begin{equation} \mathcal{F}(t) \ := \ \left(1 - \frac{1}{t}\right) \prod_{i = 0}^{\infty} \left(1 + \frac{1}{t^{3^{i}}}\right).\end{equation}
The asymptotic density of the greedy set $G_{3,q}^* \subseteq \F_q[x]$, denoted $d(G_{3,q}^*)$, is
\begin{align}
d(G_{3,q}^*) &\ = \ \prod_{\substack{P \in I_q }}  \mathcal{F}(N(P)) \label{densfirst} \\
&\ = \ \frac{1}{\zeta_q(2)} \prod_{i = 1}^{\infty} \frac{\zeta_q(3^i)}{\zeta_q(2 \cdot 3^i)} \label{denssec}\\
&\ = \ \left(1 - \frac{1}{q}\right) \prod_{i = 1}^{\infty} \prod_{n = 1}^{\infty} \left(1 + q^{-3^in}\right)^{m(n,q)} \label{densthird} \\
&\ = \ \left(1 - \frac{1}{q}\right) \prod_{i = 1}^{\infty} \frac{1-q^{1-2\cdot 3^i}}{1-q^{1-3^i}}. \label{densfourth}
\end{align}
\end{thm}

\begin{proof}
Because inclusion in the set $G^*_{3,q}$ is determined only by the powers of the irreducible elements in the factorization of its elements, the Chinese remainder theorem implies that the density of this set is given by an Euler product over all of the irreducible polynomials of $\F_q[x]$.  The contribution in the product of each irreducible polynomial, $P(X)$, is the density of those polynomials exactly  divisible by an admissible power of $P(x)$, given in \eqref{Equation:PrimeDensity}, which we abbreviate as $\mathcal{F}(N(P))$. The resulting Euler product is the first line  \eqref{densfirst} above. 

The second line \eqref{denssec} follows from some algebraic manipulation:
\begin{align}
d(G_{3,q}^*) &\ = \ \prod_{P \in I_q} \left[\left(1 - \frac{1}{N(P)}\right) \prod_{i = 0}^{\infty} \left(1 + \frac{1}{N(P)^{3^{i}}}\right)\right]\nonumber\\
&\ = \ \prod_{P\in I_q} \left[\left(1 - \frac{1}{N(P)^2}\right)  \prod_{i = 1}^{\infty} \left( \frac{1 - \tfrac{1}{N(P)^{2 \cdot 3^i}}}{1 - \tfrac{1}{N(P)^{3i}}}\right) \right]\nonumber\\
&\ = \ \frac{1}{\zeta_q(2)}\prod_{i = 1}^{\infty} \prod_{P \in I_q} \left( \frac{1 - {N(P)^{-2 \cdot 3^i}}}{1 - {N(P)^{-3^i}}}\right)\nonumber\\
&\ = \ \frac{1}{\zeta_q(2)}\prod_{i = 1}^{\infty} \frac{\zeta_q(3^i)}{\zeta_q(2 \cdot 3^i)}. \label{Equation:ZetaGreedyDens}
\end{align}

The third \eqref{densthird} and fourth \eqref{densfourth} equations follow by expanding \eqref{denssec} using the two representations for $\zeta_q(s)$ in  Lemma \ref{degreesproduct}, and the fact that $1/\zeta_q(2) \ = \ {(1 - q^{-1})}$:
\begin{align}
d(G_{3,q}^*) &\ = \  \frac{1}{\zeta_q(2)}\prod_{i = 1}^{\infty} \frac{\zeta_q(3^i)}{\zeta_q(2 \cdot 3^i)}\nonumber\\
&\ = \ \left(1 - \frac{1}{q}\right) \prod_{i = 1}^{\infty}\prod_{n = 1}^{\infty} \left(\frac{1 - q^{-2\cdot 3^i n}}{1 - q^{-3^i n}}\right)^{m(n,q)}\nonumber\\
&\ = \ \left(1 - \frac{1}{q}\right) \prod_{i = 1}^{\infty}\prod_{n = 1}^{\infty}  \left(1 + q^{-3^i n}\right)^{m(n,q)}.
\end{align}
\end{proof}

We can truncate the product in the fourth expression at a suitable length to calculate $d(G_{3,q}^*)$ for various finite fields up to a few decimal places. See Table \ref{densitytable} and Figure \ref{densityfigure}.
% GO OVER THESE COMPUTATIONS W/ DR. MCNEW. Check that these are ACTUAL APPROX's.
\begin{table}[ht!]
\caption{Density of $G_{3,q}^*$, calculated to 6 decimal places.}
\label{densitytable}
\begin{tabular}{ |c | c |}
\hline
$q$ & $d(G_{3,q}^*)$ \\
\hline
2 & .648361 \\
\hline
4 & .799231 \\
\hline
8 & .888862 \\
\hline
\end{tabular} \ \ \ \ \
\begin{tabular}{ |c | c |}
\hline
$q$ & $d(G_{3,q}^*)$ \\
\hline
\ 3 & .747027 \\
\hline
\ 9 & .899985 \\
\hline
27 & .964286 \\
\hline
\end{tabular} \ \ \ \ \
\begin{tabular}{ |c | c |}
\hline
$q$ & $d(G_{3,q}^*)$ \\
\hline
\ \ 5 & .833069 \\
\hline
\ 25 & .961538 \\
\hline
125 & .992063 \\
\hline
\end{tabular}  \ \ \ \ \
\begin{tabular}{ |c | c |}
\hline
$q$ & $d(G_{3,q}^*)$ \\
\hline
\ \ 7 & .874948 \\
\hline
\ 49 & .980000\\
\hline
343 & .997093 \\
\hline
\end{tabular}\end{table}

%\[\begin{tabular}{ |c | c |}
%\hline
%$q$ & $d(G_{3,q}^*)$ for $\F_q$ \\
%\hline
%13 & .92856 \\
%\hline
%13^2 & .99411 \\
%\hline
%13^3 & .99954 \\
%\hline
%\end{tabular} \]
\begin{figure}[ht!]
\caption{Density of $G_{3,q}^*$ graphed against $q$ for $q \le 130$.}
\centerline{\includegraphics[height=12.5 em]{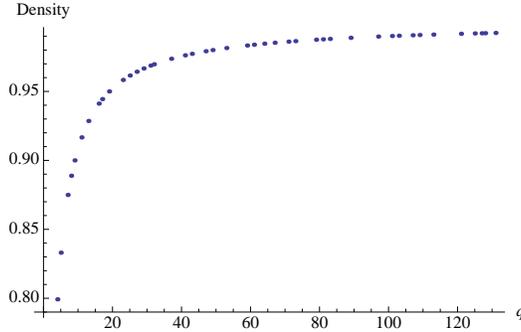}}\label{densityfigure}
\end{figure}

%%%%%%%%%%%%%%%%%%%%%%%%%%%%%%%%%%%%%%%%%%%%%%%%%%%%%%%%%%%%%%%%%%%%%%%%%%%%%%%%%%%%%%%%%%%%%%%%%%%%%%%%%%%%%%%%%%%%%%%%%%%%%%%%%%%%%%
%%%%%%%%%%%%%%%%%%%%%%%%%%%%%%%%%%%%%%%%%%%%%%%%%%%%%%%%%%%%%%%%%%%%%%%%%%%%%%%%%%%%%%%%%%%%%%%%%%%%%%%%%%%%%%%%%%%%%%%%%%%%%%%%%%%%%%
%%%%%%%%%%%%%%%%%%%%%%%%%%%%%%%%%%%%%%%%%%%%%%%%%%%%%%%%%%%%%%%%%%%%%%%%%%%%%%%%%%%%%%%%%%%%%%%%%%%%%%%%%%%%%%%%%%%%%%%%%%%%%%%%%%%%%%
\section{Lower Bounds for the Supremum of Upper Densities}\label{Section:Lowerbds}

We start with a definition of upper density (see Definition \ref{def:normle} for the definition of $S(n)$ and $S(X)$ for $X \subseteq \N$).

\begin{defn}\label{Definition:Upperdensity}
The upper density of a set $A \subseteq \F_q$ is defined as
\begin{equation} \overline{d}(A) \ = \ \limsup_{n \to \infty} \frac{|A \cap S(n)|}{|S(n)|}.\end{equation}
Since the sequence on the right hand side is bounded and monotonic, $\overline{d}(A)$ always exists.
\end{defn}

As in prior work studying this problem over other rings, we wish to study the supremum of the upper densities of 3-term non-unit geometric progression-free sets. In particular, we would like to see how much larger the upper density of such a subset of $\F_q[x]$ can be made compared to the density of the set obtained greedily.  By constructing a set of sizeable upper density we get lower bounds for this supremum.  Our method is similar to that of the construction given first in \cite{Mc} and generalized in \cite{BHMMPTW}. However, due to the very nice multiplicative structure of the possible norms in $\F_q[x]$, (i.e., they are all of the form $q^n$) we can actually do slightly better than what is obtained by simply applying the technique used there to this setting.

The idea behind this construction is to leave certain gaps between norms of elements included in the set in order to guarantee that the set will be free of some geometric progressions.

\begin{lemma} \label{Lemma:Arithgreedy}The elements of $A_3^*(\Z) = \{0,\ 1,\ 3,\ 4,\ 9,\ 10,\ 12,\ 13,\  \dots\}$, the set constructed by greedily avoiding 3-term arithmetic progressions in $\Z$, are exactly those containing no 2 in their ternary expansion
\end{lemma}

This is well known, see for example \cite{Ran}.

\begin{thm}\label{Theorem:T3q}Fix $q$ and consider the set $S(T_{3,q})$ where
\begin{equation}\label{Equation:Lowerbd}
T_{3,q} \ := \  \{ q^n : n \in A_3^*(\Z)\}.
\end{equation}
This set is free of geometric progressions and has upper density
\begin{equation}
\overline{d}(S(T_{3,q})) = m_q \ := \ \left(1-q^{-2}\right) \prod_{i = 1}^{\infty}\left(1+q^{-3^i}\right).
\end{equation}
Furthermore, no geometric-progression-free subset of $\mathbb{F}_q[x]$ of the form $S(X)$ for some $X \subset \N$ has upper density greater than $m_q$.
\end{thm}

\begin{proof}

If the elements $b(x),\ r(x)b(x),\ r(x)^2b(x) \in \F_q[x]$ are in geometric progression, then the degrees of these polynomials form an arithmetic progression $\deg(b)$, $\deg (b) + \deg(r)$, $\deg(b) + 2\deg(r)$ and, equivalently, their norms form a geometric progression, $N(b)=q^{\deg(b)}$, $N(rb) =q^{\deg(b)+\deg(r)}$, $N(r^2b) =q^{\deg(b)+2\deg(r)}$ in $\mathbb{Z}$. Since the set $T_{3,q}$ has no 3-term geometric progressions, $S(T_{3,q})$ cannot have any 3-term geometric progressions either.

We will first show $S(T_{3,q})$ has upper density at least $m_q$. For any $k \geq 1$ define

%For a number ${a_1a_2\dots a_k}_{3}$, we let the subscript 3 denote that $a_1a_2\dots a_k$ is its ternary expansion. 

\begin{equation}N_k \ := \  \sum_{i=0}^{k-1} 3^i = \frac{3^k-1}{2}, \end{equation}
%\overbrace{1 \dots 1_3}^{k \text{ times}}, \end{equation}
(note that each $N_k$ is an integer which has only 1's in its ternary expansion) and consider the quantity $|S(T_{3,q}) \cap S(q^{N_k})|/|S(q^{N_k})|$:
\begin{align}
\frac{|S(T_{3,q}) \cap S(q^{N_k})|}{|S(q^{N_k})|} &\ = \ \frac{1}{q^{N_k+1}}\sum_{\substack{n \in A_3^*(\Z)\\ n \leq N_k}} |S(q^n)| \notag \\
&\ = \ \frac{1}{q^{N_k+1}}\sum_{\substack{n \in A_3^*(\Z)\\ n \leq N_k}} \left(q^n - q^{n-1}\right) \notag \\
 &\ = \ \sum_{\substack{n \in A_3^*(\Z)\\ n \leq N_k}} \left(q^{-N_k + n} - q^{-N_k + n - 1}\right).
\end{align}

Now, the key observation is that for these choices of $N_k$, the sets $A_3^*(\Z) \cap [0,N_k]$ have an important symmetry: $A_3^*(\Z) \cap [0,N_k] \ = \ N_k - {(A_3^*(\Z) \cap [0,N_k])}$. This follows from the characterization in Lemma \ref{Lemma:Arithgreedy}. Thus
\begin{eqnarray}\label{Equation:DensityuptoNk}
\frac{|S(T_{3,q}) \cap S(q^{N_k})|}{|S(q^{N_k})|} & \ = \ & \sum_{n \in A_3^*(\Z) \cap [0,N_k]} \left(q^{-N_k + n} - q^{-N_k + n - 1}\right) \nonumber\\ &  \ = \ & \sum_{n \in A_3^*(\Z) \cap [0,N_k]} \left(q^{-n} - q^{-n - 1}\right).
\end{eqnarray}

Another way to view this is that the set $S(T_{3,q}) \cap S(q^{N_k})$ is equal to the set obtained by constructing a subset of the elements of $\mathbb{F}_q[x]$ of norm at most $q^{N_k}$ by working backwards, starting with elements of largest possible norm and greedily including those that do not form a geometric progression.

Because $S(T_{3,q}) \cap S(q^{N_k})$ is an initial subset of the set $S(T_{3,q})$, we have that
\begin{align} \label{equation:denst3q} \overline{d}(S(T_{3,q}))  & \ = \ \limsup \frac{|S(T_{3,q}) \cap S(N)|}{|S(N)} \\
& \ \ge \ \lim_{k \to \infty} \frac{|S(T_{3,q}) \cap S(q^{N_k})|}{|S(q^{N_k})|} 
 \ = \ \sum_{n \in A_3^*(\Z)} \left(q^{-n} - q^{-n - 1}\right) \nonumber \\
& \ = \ \left(1-\frac{1}{q}\right)\left(\sum_{n \in A_3^*(\Z)} q^{-n}\right)  \ = \ \left(1 - \frac{1}{q} \right) \prod_{i = 0}^{\infty} \left(1 + q^{-3^i}\right)\notag \\
&\ = \  \left(1 - \frac{1}{q^2} \right) \prod_{i = 1}^{\infty} \left(1 + q^{-3^i}\right) \ = \ m_q.\end{align}

Now we show that the upper density of $S(T_{3,q})$ is at most $m_q$ by showing that any geometric-progression free set of the form $S(X)$ has upper density at most $m_q$.  Fix $X \subset \N$ such that $S(X)$ is free of geometric progressions, and let $A = \{n \in \Z_{\geq 0}\mid q^n \in X\}$ be the set of degrees of the polynomials included in $S(X)$.  As argued before, the set of integers in $A$ must be free of arithmetic progressions in order for $S(X)$ to be free of geometric progressions.   

For any integer $M>0$ consider $A\cap [0,M]$: the set of integers in $A$ up to $M$.  Note that the number ($q^{n+1}-q^n$) of polynomials of degree $n$ in $\F_q[x]$  is greater than or equal to the number of polynomials of all lower degrees combined ($q^n$).  Thus we see that $|S(X)\cap S(q^m)|$ is maximized when $A\cap [0,M] = M - (A_3^* \cap [0,M])$, because the largest possible set is obtained by including all polynomials of degree $M$, and subsequently greedily including polynomials of the largest possible degree that can be added without introducing a geometric progression.  Thus 
\begin{align}\frac{|S(X) \cap S(q^M)|}{|S(q^m)|} \leq \frac{1}{q^{M+1}}\sum_{\substack{i \in A_3^*(\Z) \\ i\leq M}} q^{m-i+1}-q^{m-i} = \sum_{\substack{i \in A_3^*(\Z)\\i\leq M}}\left( q^{-i}-q^{-i-1}\right),\end{align}
and so 
\begin{align}
\overline{d}(S(X))   \ = \ \limsup \frac{|S(X) \cap S(q^M)|}{|S(q^M)|}  \ \leq \ \sum_{\substack{i \in A_3^*(\Z)}} \left(q^{-i}-q^{-i-1}\right) = m_q,
\end{align}
as seen in \eqref{equation:denst3q}.  Therefore, in particular, $\overline{d}(S(T_{3,q})) \leq m_q$. 
\end{proof}

As an immediate corollary, we get a lower bound on the upper densities of sets avoiding $3$-term geometric progressions in $\F_q[x]$.  Note, however, that any such set with higher upper density could not be of the form $S(X)$ for any $X \subset \N$.

\begin{cor} The supremum of the upper densities of sets avoiding $3$-term geometric progressions in $\F_q[x]$ is greater than or equal to
\begin{align} \label{Equation:LowerBound}
m_q &\ = \ \left(1-q^{-2}\right) \prod_{i = 1}^{\infty}\left(1+q^{-3^i}\right) \\
&\ = \ (1 - q^{-2}) + (q^{-3} - q^{-5}) + (q^{-9} - q^{-11}) + (q^{-12} - q^{-14}) + \cdots. \nonumber
\end{align}
\end{cor}

Some values are computed in Table \ref{lowerbdtable}.

\begin{table}[h!]
\caption{Lower bounds $m_q$ on the supremum of upper densities of subsets of $\F_q[x]$ avoiding 3-term progressions, calculated to 6 decimal places.}
\label{lowerbdtable}
\begin{tabular}{ |c | c |}
\hline
$q$ & Lower Bd \\
\hline
2 & .845398\\
\hline
4 & .952152 \\
\hline
8 & .986298 \\
\hline
\end{tabular} \ \ \ \ \
\begin{tabular}{ |c | c |}
\hline
$q$ & Lower Bd \\
\hline
\ 3 & .921858 \\
\hline
\ 9 & .989009 \\
\hline
27 & .998679\\
\hline
\end{tabular} \ \ \ \ \
\begin{tabular}{ |c | c |}
\hline
$q$ & Lower Bd \\
\hline
\ \ 5 & .96768 \\
\hline
\ 25 & .998464 \\
\hline
125 & .999937 \\
\hline
\end{tabular} \ \ \ \ \
\begin{tabular}{ |c | c |}
\hline
$q$ & Lower Bd \\
\hline
\ \ 7 & .982448 \\
\hline
\ 49 & .999592 \\
\hline
343 & .999992 \\
\hline
\end{tabular}
\end{table}

%%%%%%%%%%%%%%%%%%%%%%%%%%%%%%%%%%%%%%%%%%%%%%%%%%%%%%%%%%%%%%%%%%%%%%%%%%%%%%%%%%%%%%%%%%%%%%%%%%%%%%%%%%%%%%%%%%%%%%%%%%%%%%%%%%%%%%
%%%%%%%%%%%%%%%%%%%%%%%%%%%%%%%%%%%%%%%%%%%%%%%%%%%%%%%%%%%%%%%%%%%%%%%%%%%%%%%%%%%%%%%%%%%%%%%%%%%%%%%%%%%%%%%%%%%%%%%%%%%%%%%%%%%%%%
%%%%%%%%%%%%%%%%%%%%%%%%%%%%%%%%%%%%%%%%%%%%%%%%%%%%%%%%%%%%%%%%%%%%%%%%%%%%%%%%%%%%%%%%%%%%%%%%%%%%%%%%%%%%%%%%%%%%%%%%%%%%%%%%%%%%%%
\section{Upper Bounds for the Supremum of Upper Densities}

In order to obtain upper bounds for the upper density of a set of polynomials avoiding 3 term geometric progressions, we again use techniques analogous to those developed by Riddell\cite{Rid} over the integers.  To compute an upper bound for this upper density, we show that a certain proportion of elements must be excluded in order to avoid three term progressions, using the fact that the minimum non-unit norm is $q$.

%%%%%%%%%%%%%%%%%%%%%%%%%%%%%%%%%%%%%%%%%%%%%%%%%%%%%%%%%%%%
%%%%%%%%%%%%%%%%%%%%%%%%%%%%%%%%%%%%%%%%%%%%%%%%%%%%%%%%%%%%
%%%%%%%%%%%%%%%%%%%%%%%%%%%%%%%%%%%%%%%%%%%%%%%%%%%%%%%%%%%%
\subsection{An Upper Bound for the Upper Density}

Fix an element $r(x) \in \F_q[x]$ of norm $q$. The proportion of elements of $\F_q[x]$ that are coprime to $r(x)$ is $(q-1)/q$. Now consider $S(q^n)$. For each $b(x) \in S(q^n)$ such that $b(x)$ is coprime to $r(x)$ and $N(b) \le q^{n-2}$, we have that $b(x),\ r(x)b(x),\ r(x)^2b(x)$ forms a progression, and that all of these progressions $\{b,\ rb,\ r^2b\}$ are disjoint for different choices of $b(x)$. Thus, each $b(x)$ corresponds to a unique element (one of the elements $b(x),\ r(x)b(x),\ r(x)^2b(x)$) which must be excluded in order to avoid geometric progressions.

This can be extended further by including additional progressions. If again $b(x)$ is coprime to $r(x)$ and also $N(b) \le q^{n-5}$, then $\{br^3,\ br^4,\ br^5\}$ forms another progression. Each one of these new progressions is likewise disjoint for different choices of $b$ as well as from the previously mentioned progressions. Taking these two sets of progressions into account and letting $n \to \infty$, we get an upper bound of
\begin{equation} \lim_{n \to \infty} 1 - \left(\frac{q-1}{q}\right) \left(\frac{\left|S(q^{n-2})\right| + \left|S(q^{n-5})\right|}{\left|S(q^{n})\right|}\right) \ = \ 1 - \left(\frac{q-1}{q}\right) (q^{-2} + q^{-5})\end{equation}
for the supremum of upper densities of geometric progression free subsets of $\F_q[x]$. By continuing in this way, including additional non-overlapping progressions we can improve this bound to
\begin{equation}\label{Equation:Upperbd1}
 1 - \left(\frac{q-1}{q}\right) \left(\frac{1}{q^2}\right) \left(\sum_{i  = 0}^{\infty} \frac{1}{q^{3i}} \right)\ = \ 1 - \left(\frac{q-1}{q^3}\right)\left(\frac{1}{1 - \frac{1}{q^3}}\right) \ = \ 1 - \frac{q-1}{q^3 - 1}.
\end{equation}

\sloppy{We remark that this can be improved further by applying the methods of Nathanson and O'Bryant \cite{NO2}.  In particular, one observes that removing one out of every three consecutive integers is still insufficient to avoid arithmetic progressions in the exponents of $r$, and thus to avoid progressions involving $r$.  Taking into account all progressions involving powers of the element $r$ as in \cite{NO} we obtain the upper bound
\begin{equation}\label{Equation:NObound}
 (q-1)\sum_{n=1}^\infty \frac{1}{q^{r_n}},
\end{equation} 
where $r_n$ is the least integer such that there exists a subset of $[1,r_n]$ of size $n$. (The sequence of $r_n$ begins $(r_n)=  	1, 2, 4, 5, 9, 11, 13, 14, 20, \ldots$)}

This results in a small numerical improvement in the upper bounds.  Some values are computed in Table \ref{upperbd1table}.

\begin{table}[ht!]
\caption{Upper bounds obtained using \eqref{Equation:Upperbd1} and 
\eqref{Equation:NObound} on the supremum of upper densities of subsets of 
$\F_q[x]$ avoiding 3-term progressions, rounded to 9 decimal places, and compared to 
the lower bounds from \eqref{Equation:LowerBound}.}
\label{upperbd1table}
\begin{tabular}{ |c | c | c| c| }
\hline
$q$ & Upper Bd \eqref{Equation:Upperbd1} & Upper Bound  \eqref{Equation:NObound} & Lower Bound \eqref{Equation:LowerBound}\\
\hline
2 & 0.857142857 & 0.846375541 & 0.845397956 \\
\hline
3 & 0.923076923 & 0.921925273 & 0.921857532 \\
\hline
4 & 0.952380952 & 0.952160653 & 0.952152070 \\
\hline
5 & 0.967741935 & 0.967682134 & 0.967680495 \\
\hline
7 & 0.982456140 & 0.982447941 & 0.982447814 \\
\hline
8 & 0.986301370 & 0.986297660 & 0.986297615 \\
\hline
9 & 0.989010989 & 0.989009149 & 0.989009131 \\
\hline
11 & 0.992481203 & 0.992480647 & 0.992480643 \\
\hline
13 & 0.994535519 & 0.994535314 & 0.994535313 \\
\hline
16 & 0.996336996 & 0.996336937 & 0.996336937 \\
\hline
17 & 0.996742671 & 0.996742630 & 0.996742630 \\
\hline
19 & 0.997375328 & 0.997375307 & 0.997375307 \\
\hline
23 & 0.998191682 & 0.998191675 & 0.998191675 \\
\hline
25 & 0.998463902 & 0.998463898 & 0.998463898 \\
\hline
 \end{tabular}
\end{table}

%%%%%%%%%%%%%%%%%%%%%%%%%%%%%%%%%%%%%%%%%%%%%%%%%%%%%%%%%%%%%%%%%%%%%%%%%%%%%%%%%%%%%%%%%%%%%%%%%%%%%%%%%%%%%%%%%%%%%%%%%%%%%%%%%%%%%%
%%%%%%%%%%%%%%%%%%%%%%%%%%%%%%%%%%%%%%%%%%%%%%%%%%%%%%%%%%%%%%%%%%%%%%%%%%%%%%%%%%%%%%%%%%%%%%%%%%%%%%%%%%%%%%%%%%%%%%%%%%%%%%%%%%%%%%
%%%%%%%%%%%%%%%%%%%%%%%%%%%%%%%%%%%%%%%%%%%%%%%%%%%%%%%%%%%%%%%%%%%%%%%%%%%%%%%%%%%%%%%%%%%%%%%%%%%%%%%%%%%%%%%%%%%%%%%%%%%%%%%%%%%%%%
\section{Future Work}

Naturally, one would like to be able to further tighten these bounds by either finding sets avoiding geometric progressions with larger densities and upper densities or by improving the upper bounds.  We end with a few ideas for how this might be done, and additional areas to explore.  

In the integer case \cite{Mc} it has been shown that it is possible to construct sets with higher upper density than the greedy set (but only when avoiding progressions with integer ratio).  In particular, by looking at the small primes 2, 3 and 5, it is possible to obtain a small improvement on the greedy set by removing those elements with only a single power of 2, 3 or 5, which allows other elements to be included later on.  It seems possible that a similar construction might yield a set with higher asymptotic density in $\F_q[x]$. 

\begin{question} Can a set be constructed to avoid 3-term progressions in $\F_q[x]$ with density greater than $G_{3,q}^*$? If so, can this be done for all $q$ or only for $q$ small enough/large enough?
\end{question}

In previous work, \cite[Section 4]{Mc}, \cite[Section 4.2]{BHMMPTW} in different rings, improved upper bounds for the upper density were obtained by considering progressions among the smooth integers.  Thus far this technique has not proven to be as useful in this ring, however it may just be that more work and computation are required.  

\begin{question} Can bounds for progressions among the smooth polynomials (those without any irreducible factors of high degree) be used to obtain better upper bounds for the upper density of a subset of $\F_q[x]$ avoiding 3-term geometric progressions?
\end{question}

In general, one might even hope that the additional combinatorial structure of the ring $\F_q[x]$ might allow one to work out exactly what the supremum of the upper densities of sets avoiding 3-term progressions is, even though this appears to be far more elusive over the integers. From the tightness of our bounds and the natural construction of $S(T_{3,q})$, we propose:

\begin{conjecture} The supremum of upper densities of sets avoiding 3-term geometric progressions in $\F_q[x]$ is $m_q$.
\end{conjecture}

A similar trick to the second half of Theorem \ref{Theorem:T3q} proves useful: if the upper density up to some norm $q^N$ of a set $S$ is higher than $m_q$, then there must be some $q^B$ such that all polynomials of norm $q^B$ are in $S(T_{3,q})$ and some, but not all, are in $S$. Then there is some factorization pattern of elements of norm $q^B$ that are all not excluded. Most likely, the divisors of the factorization will either have too many geometric progressions or there will be too few to begin with.

It appears likely that considering progressions among the smooth polynomials would be far more useful in $\F_q[x]$ for studying subsets that avoid progressions with rational ratio (where one allows the geometric progressions to have common ratio in $\F_q(x)$ rather than just $\F_q[x]$).  This version of the problem has also been well studied over the integers, however we have chosen not to consider such progressions here.  This would be interesting to consider in further work as well.

\begin{question} How does the situation described here change if one allows geometric progressions with ratio in $\F_q(x)$? Note that the greedy set $G_{3,q}^*$ constructed here avoids such progressions as well, but the set $S(T_{3,q})$ does not.  Are there examples of sets avoiding such rational-ratio progressions with greater upper density than $G_{3,q}^*$?
\end{question}

\begin{question} What happens if the field $\F_q$ is replaced by a different finite ring?
\end{question}
%%%%%%%%%%%%%%%%%%%%%%%%%%%%%%%%%%%%%%%%%%%%%%%%%%%%%%%%%%%%%%%%%%%%%%%%%%%%%%%%%%%%%%%%%%%%%%%%%%%%%%%%%%%%%%%%%%%%%%%%%%%%%%%%%%
%%%%%%%%%%%%%%%%%%%%%%%%%%%%%%%%%%%%%%%%%%%%%%%%%%%%%%%%%%%%%%%%%%%%%%%%%%%%%%%%%%%%%%%%%%%%%%%%%%%%%%%%%%%%%%%%%%%%%%%%%%%%%%%%%%

\end{document}